\documentclass{article}
\usepackage{CJKutf8}

\newtheorem{fed}{\textbf{Definition}}[section]
\newtheorem{thm}[fed]{\textbf{Theorem}}

\newtheorem{lemma}[fed]{\textbf{Lemma}}

\newtheorem{rem}[fed]{\textbf{Remark}}
\newtheorem{prop}[fed]{\textbf{Proposition}}
\newtheorem{cor}[fed]{\textbf{Corollary}}
\usepackage{amscd,amssymb,bbm,graphicx,epsfig,psfrag,epic,eepic,latexsym,amsmath}
\usepackage{amsmath}
\usepackage[arrow,matrix,curve]{xy}
\usepackage{mathrsfs}
\usepackage[normalem]{ulem}
\usepackage{soul}
\usepackage
{color}

\begin{document}

\title{The kei word metric on the transvection group and its unit ball}
\author{Urs Frauenfelder, Lei Zhao}
\maketitle

\abstract
We define a bi-invariant word metric on a group using keis. We discuss whether such a word metric is trivial or non-trivial in various examples coming from group theory, dynamical systems, as well as complex and symplectic geometry.

\section{Introduction}

\subsection{Historical remarks about ''kei``}
In its 1943 volume (the last one before it resumed in 1949), the Tohoku Mathematical Journal published  a paper by Mituhisa Takasaki entitled ''Abstractions of
symmetric functions``. Although it has this English title besides a Japanese one, the paper is written in an oldstyle Japanese which is challenge to read even for modern
Japanese readers. 

About the life of this mathematician, we basically do not know anything. From this article, it seems that he was working at Harbin which was then part of the Japanese puppet Manchukuo.  See for example \cite{otsuka} for the difficult historical circumstances which
probably overshadowed the life of Takasaki. In his paper he mentioned that he plans a sequel which however never seemed to appear and we are not aware of any other paper by Takasaki. 

In his paper, Takasaki discovered an algebraic structure one encounters when studying involutions. What motivated him
to refer to this structure as a ''kei`` (''\begin{CJK}{UTF8}{gbsn}圭\end{CJK}``) remains mysterious to us. 

The structure of a kei naturally appears in the study of
symmetric spaces, where the geodesic reflections at all the points of a symmetric space are all involutions, and give rise to the structure of a kei.  This kei structure plays a crucial role in Loos' \cite{loos} approach to symmetric spaces.

However, we point out that the notion of  kei
is more general than the notion of a symmetric space structure and allows infinite dimensional examples as well. For example{,}
the space of all antisymplectic involutions of a symplectic manifold is a kei, while the generalization of symmetric spaces to
infinite dimensions is still a matter of ongoing debate \cite{freyn-hartnick-horn-koehl}. Through the work of Loos{,} the notion
of {kei} was rediscovered under different names like ''symmetric set`` in the work of Nobusawa \cite{nobusawa} or
''symmetric groupoids`` in the work of Pierce \cite{pierce}. In his foundational paper on quandles, Joyce rediscovered an english
review of the paper of Takasaki \cite{joyce}, and since then a kei is also known under the name of an \emph{involutive quandle}.
For more information on the history of involutive quandles we refer to the article of Stanovsk\'y \cite{stanovsky}. 

\subsection{Bi-invariant metric given by a kei}
Following standard terminology of the theory of symmetric space, see for example \cite{loos}, we consider the transvection group
of a kei. We explain how the kei endowed its transvection group with a bi-invariant metric. Bi-invariant metrics on the
Hamiltonian diffeomorphism group of a symplectic manifold play an important role in Hamiltonian dynamics. A famous example
is the Hofer metric, see for example \cite{hofer, hofer-zehnder, mcduff-salamon, polterovich}. However, there are other
bi-invariant metrics of dynamical origin on it, see for example \cite{brandenbursky-kedra-shelukhin, burago-ivanov-polterovich}.

The kei metric and particularly its unit ball is related to the question of reversibility. Questions of reversibility arise in
many different areas of mathematics, like group theory, algebraic geometry, ergodic theory and the theory of dynamical systems.
We refer to the book by O'Farrell and Short \cite{ofarrell-short}, and the literature therein on this huge
subject.  

The same group can arise as the transvection group of different keis which can endow the transvection group with different
kinds of bi-invariant metrics. This parallels the phenomenon that in the theory of reversibility one can fine tune the question
of reversibility to the question for which kind of transformations a given map is reversible. For example in the case
of Hamiltonian diffeomorphisms one can ask if they are reversible with respect to symplectic or antisymplectic involutions. 

The difference between symplectic and antisymplectic involutions is reminiscent of old questions in geometry. Inspired by Hilbert's famous saying ''one must be able to say at all times--instead of points, straight lines, and planes--tables, chairs, and beer mugs`` \cite{blumenthal}, it was Hjelmslev \cite{hjelmslev} who started to identify points with the reflection
at the point and line with reflection at the line. This allows to translate many geometric properties to algebraic properties.
For example, with this viewpoint, a point is coincident with a line if they commute, or two lines are perpendicular to each other if they 
commute. Detailed explanations of this intriguing dictionary between algebra and geometry can be found in Bachmann's book
''Aufbau der Geometrie aus dem Spiegelungsbegriff`` \cite{bachmann1}, see also \cite{bachmann}. 

It seems that the work of
Bachmann and the work of Takasaki are closely connected. The only reference Takasaki mentions in his paper is
Thomsen's ''Grundlagen der Elementargeometrie``, \cite{thomsen}. On the other hand this work of Thomsen plays as well
an important role in Bachmann's book \cite{bachmann1}. In particular, the ''Thomsen relation`` introduced in
\cite{bachmann1} is closely connected to reversibility. 

In the book of Bachmann \cite{bachmann1} a general reduction theorem is proved which tells us that in many plane geometries
isometries can be written as a product of two lines, i.e., two reflections at lines. In Euclidean geometry, this was already proved in the 19th century by Wiener \cite{wiener}. Why it holds in spherical and hyperbolic geometry is for example nicely explained in the book of Needham \cite[Chapter 6]{needham}. What about points, respectively reflections at points? 
In Euclidean geometry, the composition of two point reflections leads to a translation and the group of translations
is a proper normal subgroup of the group of orientation preserving isometries of the Euclidean plane. The situation in
hyperbolic geometry contrasts sharply with the euclidean one. Different from the orientation preserving euclidean isometries, the
orientation preserving hyperbolic isometries build a simple group. For that reason,  any orientation preserving hyperbolic 
isometry can be obtained as an even product of point reflections. In the language we introduce{,} the minimal number
of pairs of point reflections needed to obtain a given orientation preserving hyperbolic isometry is its norm with respect
to the kei metric associated to point reflections. However, different from the case of lines two points reflections are in general not sufficient to get an orientation preserving hyperbolic isometry. How many point reflections you need, is studied
by Kn\"uppel \cite{knueppel} and Stroscher \cite{stroscher} in the context of Bachmann and Hjelmslev groups.  

With this viewpoint, we may regard symplectic involutions as points and anti-symplectic involutions as lines. In this setting the corresponding geometric problems are translated into the trivialness of the kei-metrics.

We mention that bi-invariant metrics using the idea of racks and quandles have also been explored in \cite{Kedra}.
\\ \\
\emph{Acknowledgements: } This research has been carried out during 2016-2017. We thank Peter Quast, Makiko Tanaka for very useful discussions on symmetric spaces. The first named author gave a talk during the conference ``Topological data analysis meets symplectic topology'' in Tel Aviv in 2018.  We are happy that this talk raised interest from Egor Shelukhin and Leonid Polterovich, with whom we had intensive discussions over the years, which results in the joint work \cite{FPSZ}.

\section{The word metric}\label{Sec: The word metric}

We recall in this section how to associate to a symmetric, conjugation invariant generating set of a group a bi-invariant metric, following the article of Gal and Kedra \cite{gal-kedra}.
Suppose that $G$ is a group and $S \subset G$ is a generating set which is symmetric in the sense that $S=S^{-1}$, i.e., $g \in S$ if and only if $g^{-1} \in S$. Then for an element $g \in G$ the \emph{word norm} with respect to $S$ is defined as
$$|g|_S:=\min\{k \in \mathbb{N}_0: g=s_{i_1}\cdots s_{i_k}, s_i \in S\}.$$
The word norm has the following properties.
\begin{description}
 \item[(i)] $|g|_S \geq 0$ and $|g|_S=0$ if and only if $g=\mathrm{id}.$
 \item[(ii)] $|gh|_S \leq |g|_S+|h|_S$, for $g,h \in G$.
 \item[(iii)] $|g^{-1}|_S=|g|_S.$
\end{description}
The last property follows from the assumption that the generating set is symmetric. One then defines a metric on $G$ by setting for $g,h \in G$
$$d_S(g,h):=|g h^{-1}|_S.$$
This metric is right-invariant. In addition, if the generating set $S$ is invariant under conjugation, then the metric is additionally left invariant and therefore bi-invariant. 

We abbreviate by
$$B_S:=\{g \in G:|g|_S \leq 1\}=\{g \in G: d_S(g,\mathrm{id}) \leq 1\}$$
the unit ball in $G$ centered at the identity. 

Note that the following three properties are equivalent.
\begin{description}
 \item[(i)] The metric $d_S$ is \emph{trivial} in the sense that
$$d_S(g,h)=\left\{\begin{array}{cc}
1 & g \neq h\\
0 & g=h
\end{array}\right.$$
 \item[(ii)] $B_S=G$.
 \item[(iii)] $S=G \setminus \{\mathrm{id}\}$.
\end{description}

\section{The kei word metric}

Let $X$ be a set and denote by $\mathrm{Inv}(X)$ the set of all involutions of $X$, i.e., the set of all
maps $R \colon X \to X$, satisfying $R^2=\mathrm{id}|_X$. Conjugation defines a product on
$\mathrm{Inv}(X)$
$$R \star S= R S R.$$
This product has the following properties
\begin{description}
 \item[(i)] $R \star R=R.$
 \item[(ii)] $R \star (R \star S)=S.$
 \item[(iii)] $R \star (S \star T)=(R \star S) \star (R \star T).$
\end{description}
A tuple $(\mathcal{K}, \star)$ consisting of a set $\mathcal{K}$ endowed with a multiplication
$\star \colon \mathcal{K} \times \mathcal{K} \to \mathcal{K}$ satisfying properties (i)--(iii), 
is precisely an \emph{involutive quandle \cite{carter}}, or a \emph{Kei} as defined initially by Takasaki \cite{takasaki}. Given a kei $(\mathcal{K},\star)$ we
get by property (ii) a map
$$\mathcal{L} \colon \mathcal{K} \to \mathrm{Inv}(\mathcal{K}), \quad R \mapsto \mathcal{L}_R$$
where $\mathcal{L}_R$ is the left multiplication of $R$ on $\mathcal{K}$
$$\mathcal{L}_R \colon \mathcal{K} \to \mathcal{K}, \quad S \mapsto R \star S.$$
If $R, S, T \in \mathcal{K}$ are three elements of our kei, then we compute using properties
(ii) and (iii)
\begin{eqnarray*}
\mathcal{L}_{R} \mathcal{L}_{S} \mathcal{L}_{R}(T)&=&
R \star (S \star (R \star T))\\
&=&(R \star S) \star (R \star (R \star T))\\
&=&(R \star S) \star T\\
&=&\mathcal{L}_{R \star S}(T).
\end{eqnarray*}
Hence if we define a \emph{subkei} of a kei as a subset closed under multiplication{,} then the computation above shows that the image of the map $\mathcal{L}$ is a subkei of $\mathrm{Inv}(\mathcal{K})$.

For a set $X$ abbreviate by
$$\mathfrak{S}(X):=
\{\phi \colon X \to X, \phi\,\,\mathrm{bijective}\}$$
the group of permutations of $X$. 
For a kei $\mathcal{K}$ the \emph{transvection group}
$$G(\mathcal{K}) \subset \mathfrak{S}(\mathcal{K})$$
is defined as
$$G(\mathcal{K})=\langle \mathcal{L}_R \mathcal{L}_S: R,S \in 
\mathcal{K} \rangle,$$
i.e. the subgroup of perturbations generated by products of left multiplication maps of elements of the kei. 
Because $\mathcal{L}_S$ and $\mathcal{L}_R$ are involutions it holds that
$$(\mathcal{L}_S \mathcal{L}_R)^{-1}=\mathcal{L}_R \mathcal{L}_S$$
and therefore the generating set of the transvection group consisting of all products $\mathcal{L}_R \mathcal{L}_S$
is symmetric. Moreover, this set is also conjugation invariant since
$$\mathcal{L}_{K} \mathcal{L}_R \mathcal{L}_S \mathcal{L}_K=\mathcal{L}_{K \star R} \mathcal{L}_{K \star S}.$$
 Therefore it defines a bi-invariant metric
$$d_\mathcal{K} \colon G(\mathcal{K}) \times G(\mathcal{K}) \to \mathbb{N}_0$$
on the transvection group which we refer to as the \emph{kei word metric}.

\begin{rem} The generating set $\mathcal{S}=\{\mathcal{L}_{R} \mathcal{L}_{S}: R, S \in \mathcal{K} \}$ is a kei as well, with the conjugation as kei operation.
Indeed, we have
\begin{eqnarray*}
\mathcal{L}_{R} \mathcal{L}_{S} \mathcal{L}_{R'} \mathcal{L}_{S'} \mathcal{L}_{S} \mathcal{L}_{R}&=&
\mathcal{L}_{R} \mathcal{L}_{S} \mathcal{L}_{R} \mathcal{L}_{R} \mathcal{L}_{R'} \mathcal{L}_{R} \mathcal{L}_{R} \mathcal{L}_{S'} \mathcal{L}_{R} \mathcal{L}_{R} \mathcal{L}_{S} \mathcal{L}_{R}\\
&=& \mathcal{L}_{R \star S} \mathcal{L}_{R \star R'} \mathcal{L}_{R \star S'} \mathcal{L}_{R \star S}\\
&=&\mathcal{L}_{R \star (S \star R')} \mathcal{L}_{R \star (S' \star S)}.
\end{eqnarray*}
So that the set $\mathcal{S}$ is closed under the operation by conjugation. That this operation is a kei operation follows from the previous discussions.

\end{rem}

\section{A special class of keis}\label{special}

Suppose that for a set $X$ 
$$G < \mathfrak{S}(X)$$
is a subgroup of the permutation group of $X$
and 
$$R \in \mathrm{Inv}(X)$$
is an involution satisfying
$$RGR=G,$$
i.e., $R$ is  not necessarily an element of $G$ but every element of
$G$ when conjugated with $R$ in $\mathfrak{S}(X)$ is again in $G$.
For $g \in G$ abbreviate
$$R_g:=g R g^{-1}.$$ 
We introduce the set
$$\mathcal{K}_{R,G}:=\{R_g: g \in G\}.$$
\begin{lemma}
$\mathcal{K}_{R,G}$ is a subkei of $\mathrm{Inv}(X)$. 
\end{lemma}
\textbf{Proof:} Because an involution after conjugation remains an involution, it is clear that $\mathcal{K}_{R,G}$ is a subset of
$\mathrm{Inv}(X)$. It remains to show that it is invariant under conjugation. Hence we compute that for any $g, h \in G$, it holds 
\begin{eqnarray*}
R_g R_h R_g&=& (g R g^{-1})(h R h^{-1})(g R g^{-1})\\
&=&(g R g^{-1} h R)R(R h^{-1} g R g^{-1})\\
&=&(g R g^{-1} h R)R(g R g^{-1} h R)^{-1}.
\end{eqnarray*}
Because $RGR=G$ it holds that 
$$g R g^{-1} h R \in G$$
so that 
$$R_g R_h R_g \in 
\mathcal{K}_{R,G}.$$ 
This finishes the proof of the lemma. \hfill $\square$
\\ \\
There is a canonical homomorphism
$$\Psi \colon G \to \mathfrak{S}(\mathcal{K}_{G,R})$$
given for $g,h \in G$ by
$$\Psi(g)(R_h)=R_{gh}.$$
We further introduce the subgroup
$$G_R < \mathfrak{S}(X)$$
which is generated by all products of elements in the kei
$\mathcal{K}_{R,G}$, namely
$$G_R=\langle R_g R_h: g,h \in G\rangle.$$
\begin{lemma}\label{gr}
The group $G_R$ is a normal subgroup of $G$ and the homomorphism $\Psi$ maps $G_R$ to the 
transvection group $G(\mathcal{K}_{G,R})$.
\end{lemma}
\textbf{Proof:} Generators of the group $G_R$ are of the form
$$R_g R_h=g R g^{-1} h R h^{-1}$$
with $g,h \in G$ which lie in $G$ because $RGR=G$. Therefore $G_R$ is a subgroup of $G$. That it is normal follows from the identity
$$kR_g R_h k^{-1}=R_{kg}R_{kh}$$
for three group elements $g,h,k \in G$. 
To understand the image of $\Psi$ restricted to $G_R$ we compute for $g,h,k \in G$ 
$$\Psi(R_g R_h)(R_k)=R_{R_g R_h k}=\mathcal{L}_{R_g}\mathcal{L}_{R_h}(R_k)$$
which means that $\Psi$ maps generators of $G_R$ to generators
of $G(\mathcal{K}_{G,R})$. That proves the second assertion of the lemma. \hfill $\square$

\begin{prop}\label{tra}
Assume that the group $G$ is simple and $R$ is not commuting with every element of $G$. Then it holds that $G=G_R$ and $\Psi$ maps
$G_R$ isomorphically to the transvection group of $\mathcal{K}_{G,R}$. In particular, 
$G$ is canonically isomorphic to $G(\mathcal{K}_{G,R})$. 
\end{prop}
\textbf{Proof:} By Lemma~\ref{gr}, the group $G_R$ is a normal subgroup
of $G$, thus either $G_R=G$ or $G_R$ is trivial by the simplicity of $G$. To show that it is not trivial we use the assumption that there exists $g \in G$ which does not commute with $R$. This can be rephrased as for such a choice of $g \in G$,
$$R_g \neq R$$ or equivalently
$$R_g R \neq \mathrm{id}.$$
However, $R_g R$ belongs to $G_R$. Therefore $G_R$ is not trivial and we have
$$G_R=G.$$
Invoking once more Lemma~\ref{gr} it holds that
$$\Psi|_{G_R}: G_R \to G(\mathcal{K}_{G,R})$$
is a surjective homomorphism. It remains to show that it is injective.
The kernel of a homomorphism is a normal subgroup and therefore because
$G_R=G$ is simple the homomorphism is either injective or trivial. Hence we are left with showing that it is not trivial. We use again an element $g \in G$ which does not commute with $R$ to compute
$$\Psi(g)(R)=R_g \neq R$$
which implies that
$$\Psi(g) \neq \mathrm{id}_{\mathfrak{S}(\mathcal{K}_{G,R})}.$$
This proves that the restriction of $\Psi$ to $G=G_R$ is not trivial and hence the proposition follows. \hfill $\square$
\\ \\
Now assume that $G<\mathfrak{S}(X)$ is a simple group. If $R \in \mathrm{Inv}(X)$ is an involution such that
$RGR=G$ but $R$ does not commute with every element of $G$, we can by Proposition~\ref{tra} identify $G$ with the transvection group of the kei $\mathcal{K}_{G,R}$. In particular, $R$ endows $G$ with a bi-invariant metric, namely the kei word metric 
$$d_R:=d_{\mathcal{K}_{R,G}}\colon G \times G \to \mathbb{N}_0.$$
We abbreviate by
$$B_R:=B_{\mathcal{K}_{G,R}} \subset G$$
its unit ball.
\begin{prop}\label{nontri}
Under the assumptions of Proposition~\ref{tra} if $g \in B_R$, then $g$ is 
conjugated to its inverse $g^{-1}$ by an element $R_h$ for $h \in G$.
\end{prop}
\textbf{Proof:} An element $g \in G$ lies in the unit ball if and only if there exist elements $h,k \in G$ such that
$$g=R_k R_h.$$
Because $R_k$ and $R_h$ are involutions it follows that
$$g^{-1}=R_h R_k=R_h R_k R_h R_h=R_h g R_h.$$
This proves the proposition. \hfill $\square$ 
\\ \\
Elements of a group which are conjugated to its inverse are referred to as \emph{reversible} elements. If the conjugation can
be chosen to be an involution then they are called \emph{strongly reversible}. Alternatively strongly reversible elements
are also referred to as strongly real, strictly real or bireflectional, see \cite{ofarrell-short}. In particular, we have the
following immediate Corollary:
\begin{cor}
If the kei metric is trivial, then all elements of $G$ are strongly reversible.
\end{cor}
{It is also interesting to examine how the kei metric behaves with respect to powers. Since
it is a norm, we have $|g^2| \leq 2|g|$ for any element $g \in G$. However, we can improve this estimate.
\begin{lemma}\label{squ}
Suppose that $g \in G$. Then we have the estimate
$$|g^2| \leq 2|g|-1.$$
\end{lemma} 
\textbf{Proof:} We abbreviate $n=|g|$. Then there exists involutions $R_i, S_i$ for
$1 \leq i \leq n$ all conjugated to $R$ via elements in $G$ such that
$$g=R_1 S_1R_2S_2 \cdots R_nS_n.$$
We  write
\begin{eqnarray*}
g^2&=&R_1 S_1R_2S_2\cdots R_nS_nR_1S_1R_2S_2\cdots R_nS_n\\
&=&(R_1 S_1 R_1)(R_1R_2R_1)(R_1 S_2 R_1) \ldots (R_1 R_n R_1)(R_1 S_n R_1) S_1 R_2 S_2\ldots R_n S_n\\
&=&(R_1 \star S_1)(R_1 \star R_2)(R_1 \star S_2)\ldots (R_1 \star R_n) (R_1\star S_n)
S_1 R_2 S_2 \ldots R_nS_n.
\end{eqnarray*}
It follows that
$$|g^2| \leq 2n-1=2|g|-1.$$
This proves the Lemma. \hfill $\square$
\\ \\
An immediate Corollary of the Lemma is that the unit ball is invariant under taking squares.
In fact this is true for every power. 
\begin{lemma}\label{uniball}
Assume that $g \in B_R$. Then $g^n \in B_R$ for every $n \in \mathbb{Z}$.
\end{lemma}
In order to prove this lemma we first introduce some general notation. If
$(\mathcal{K},\star)$ is a kei we define the iterated product 
$$\star_k \colon \mathcal{K} \times \mathcal{K} \to \mathcal{K}$$
for $k \in \mathbb{N}$ recursively by the requirement that $\star_1=\star$ and 
for $S,T \in \mathcal{K}$
$$S \star_{2k} T=T \star (S \star_{2k-1} T), \quad S \star_{2k+1} T=S \star(S \star_{2k} T)$$
so that we have for instance
$$S \star_2 T=T \star (S \star T), \quad S \star_3 T=S \star (T \star (S \star T)).$$
\textbf{Proof of Lemma~\ref{uniball}: } Suppose that $g \in B_R$. This means that there exists involutions 
$S$ and $T$ conjugated to $R$ by elements in $G$, such that
$$g=ST.$$
If $n=2m$ for $m \in \mathbb{N}$ we have that
$$g^n=(ST)^n=(S \star_{n-1} T)T$$
and if $n=2m+1$ for $m \in \mathbb{N}$ is an odd integer bigger or equal to three we have that
$$g^n=(ST)^n=(T \star_{n-1} S)T.$$
We see from this that
$$g^n \in B_R, \quad n \in \mathbb{N}.$$
Since $g^{-1}=TS$ lies as well in the unit ball $B_R$ we obtain by applying the above argument
to $g^{-1}$ that $g^n$ lies in the unit ball for negative powers as well. This proves the lemma.
\hfill $\square$
\begin{rem}
It should be possible to improve Lemma~\ref{squ} to the statement
$$|g^n| \leq |n|\cdot |g|-|n|+1, \quad n \in \mathbb{Z}$$
from which Lemma~\ref{uniball} is an immediate consequence. 
\end{rem}
}

\section{Alternating groups}
In this section, we illustrate the kei-metric with the alternating groups.

Consider the case when $X$ consists of n points and denote by $\mathfrak{A}_n$ the alternating group. Note that  $\mathfrak{A}_n$ is simple for any $n \ge 3, n \neq 4$.

Now for any transposition $T \in \mathfrak{S}_n$ we have 
$$T\, \mathfrak{A}_n \,T =\mathfrak{A}_n$$
since conjugation by $T$ does not change the parity of a permutation. 

The set $\mathcal{K}_{R,G}$, with $R=T, G=\mathfrak{A}_n$ is in this case the set of all transpositions in  $\mathfrak{S}_n$. Indeed, for any conjugation $g\, T \,g^{-1}$ with $g$ a permutation, we may change the parity of $g$ by replacing it with $g \, T$, therefore we may always assume that $g$ is an even permutation. The transvection group $G_R$ is thus the group generated by even number of transpositions in $\mathfrak{A}_n$, and thus agrees with $\mathfrak{A}_n$, for any $n$. We have thus defined a kei metric on $\mathfrak{A}_n$ from any transposition of $T \in \mathfrak{S}_n$.

We know that $\mathfrak{A}_2$ is trivial. $\mathfrak{A}_3$ is cyclic of order 3, and is composed of identity, $(231)=(13)(12)$ and $(312)=(231)^{-1}=(12)(13)$. For $\mathfrak{A}_4$, if a permutation of four digits is even, then the number of cycles of even length has to be even, so that we have either two cycles of length 2, or one cycles with length 3. In both cases, these elements can be written as product of two transpositions. Thus the kei-metric thus defined on them is trivial. 
L
\begin{thm} The kei metric on $\mathfrak{A}_n$ is non-trivial for any $n >5$. 
\end{thm}
\textbf{Proof: } For $n=5$, the element $(23154)$ is not a product of two transpositions, since two transpositions at most move 4 digits. Likewise, the element $(231546\cdots n)$ in $\mathfrak{A}_n, n > 5$, which fixes all but the first five digits also cannot be a product of two transpositions. 
\hfill $\square$ 

We denote by $\mathfrak{A}_{\infty}=\cup_{i=3}^{\infty} A_i$. This is a simple group, and by our previous construction also carries a kei metric induced from any transposition $T \in \mathfrak{S}_{\infty}:= \cup_{i=3}^{\infty} \mathfrak{S}_i$.  For any $n$ even, the kei word norm of the element $(12)(34)\cdots((n-1)\, n)$ is clearly $n/2$. Thus $\mathfrak{A}_{\infty}$ contains elements of any kei word norm, and thus has infinite diameter with the kei word metric.

\section{Siegel upper half plane}
However, we point out, that the kei metric depends on the choice of the involution and might be trivial for one
involution while nontrivial for another. We see an example of this phenomenon in this section.

The Siegel upper half plane for $n \in \mathbb{N}$ is the space 
$$\mathcal{H}:=\mathcal{H}_n:=\{Z \in \mathbb{C}^{n\times n}: Z=Z^T, \mathrm{Im}(Z)>0\}$$
where by $\mathrm{Im}Z>0$ we mean that the imaginary part of the complex symmetric matrix $Z$ is positive definite.
For $n=1$ this is just the usual upper half plane. The linear symplectic group $\mathrm{Sp}(n)$  acts on the Siegel upper
half space as follows. Note that if we write a symplectic matrix $\Psi$ as
$$\Psi=\left(\begin{array}{cc}
A & B\\
C & D
\end{array}\right),$$
where $A,B,C,D$ are real $n \times n$ matrices, then the fact that $\Psi$ is symplectic is equivalent to the assertion that
$$\Psi^{-1}=\left(\begin{array}{cc}
D^T & -B^T\\
-C^T & A^T
\end{array}\right)$$
which implies that, by $\Psi^{-1} \Psi=\mathrm{id}$, we get
$$A^T C=C^T A, \quad B^T D=D^T B, \quad A^T D-C^T B=\mathrm{id}.$$
On the other hand, by calculating $\Psi \Psi^{-1}=\mathrm{id}$, we get
\begin{equation}\label{sympl}
A B^T=B A^T, \quad CD^T=D C^T, \quad A D^T-BC^T=\mathrm{id}.
\end{equation}
The action of the symplectic group on Siegel upper half space is given by
$$\Psi_* Z=(AZ+B)(CZ+D)^{-1}.$$
Note that if $n=1$ this is the usual action of $\mathrm{Sp}(1)=SL(2,\mathbb{R})$ on the upper half plane by
M\"obius transformations. The action is transitive and the stabilizer at $i \cdot \mathrm{id}$ is the unitary group
$U(n)$, therefore the Siegel upper half space can be identified with the homogeneous space
$\mathrm{Sp}(n)/U(n)$, see \cite[Chapter 1 \& 2]{mcduff-salamon} and \cite{siegel}.
Note that
$$\Psi_*=(-\Psi)_*$$
so that we get an action of the projective symplectic group
$$\mathrm{PSp}(n)=\mathrm{Sp}(n)/\{\pm \mathrm{id}\}$$
which is a simple group \cite[Chapter 2.3]{conway-curtis-norton-parker-wilson}. We further mention that
the Siegel upper half space can be endowed with a K\"ahler structure so that the symplectic group acts by
K\"ahler isomorphisms, i.e., isometric symplectomorphisms, which therefore preserve as well the complex structure \cite{siegel}.
We consider the anti-K\"ahler involution, i.e., the isometric antisymplectic involution
$$\mathbb{L} \colon \mathcal{H} \to \mathcal{H}, \quad Z \mapsto -\overline{Z}.$$  
In the special case of the hyperbolic plane, i.e., $n=1$, this corresponds to the reflection at a hyperbolic line, namely
the positive imaginary axis.  
Note that
$$\mathbb{L} ( \mathrm{PSp}(n)) \mathbb{L}=\mathrm{PSp}(n).$$
\begin{thm}
The kei metric $d_\mathrm{L}$ is trivial, i.e., $B_\mathbb{L}=\mathrm{PSp}(n)$.
\end{thm}
\textbf{Proof:} This is an immediate consequence of a Theorem of Wonenburger \cite{wonenburger} which says that
every linear symplectic map can be written as the product of two linear antisymplectic involutions. \hfill $\square$
\\ \\
We consider next the K\"ahler involution, i.e., the isometric symplectic involution
$$\mathbb{P} \colon \mathcal{H} \to \mathcal{H}, \quad Z \mapsto -Z^{-1}.$$
Note that this corresponds to a point reflection at the point $i \cdot \mathrm{id}$. Observe further that
$$\mathbb{P} \in \mathrm{PSp}(n)$$
so that obviously we have
$$\mathbb{P} ( \mathrm{PSp}(n)) \mathbb{P}=\mathrm{PSp}(n).$$
If $[\Phi] \in \mathrm{PSp}(n)$ for $\Phi \in \mathrm{Sp}(n)${,} the absolute value of the
trace is independent of the choice of representative{,} so that we may set
$$|\mathrm{tr}|([\Phi]):=|\mathrm{tr}(\Phi)|.$$
With this notion we have the following lemma.
\begin{lemma}
Suppose that $[\Phi] \in B_\mathrm{P} \subset \mathrm{PSp}(n)$. Then 
$$|\mathrm{tr}|([\Phi]) \geq 2n.$$
\end{lemma}
\textbf{Proof:} Because the trace is conjugation invariant we can conjugate $\Phi$ with any element of $\mathrm{Sp}(n)$
to deduce the conclusion of the lemma. 
After proper conjugation, we can assume without loss of generality that there exists a symplectic matrix
$$\Psi=\left(\begin{array}{cc}
A & B\\
C & D
\end{array}\right)$$
such that $\Phi$ is of the form
\begin{eqnarray*}
\Phi&:=&\left(\begin{array}{cc}
A & B\\
C & D
\end{array}\right)
\left(\begin{array}{cc}
0 & -\mathrm{id}\\
\mathrm{id} & 0
\end{array}\right)
\left(\begin{array}{cc}
D^T & -B^T\\
-C^T & A^T
\end{array}\right)
\left(\begin{array}{cc}
0 & -\mathrm{id}\\
\mathrm{id} & 0
\end{array}\right)\\
&=&
\left(\begin{array}{cc}
B & -A\\
D & -C
\end{array}\right)
\left(\begin{array}{cc}
D^T & -B^T\\
-C^T & A^T
\end{array}\right)
\left(\begin{array}{cc}
0 & -\mathrm{id}\\
\mathrm{id} & 0
\end{array}\right)\\
&=&
\left(\begin{array}{cc}
AC^T+BD^T & -AA^T-BB^T\\
CC^T+DD^T & -CA^T-DB^T
\end{array}\right)
\left(\begin{array}{cc}
0 & -\mathrm{id}\\
\mathrm{id} & 0
\end{array}\right)\\
&=&
-\left(\begin{array}{cc}
AA^T+BB^T & AC^T+BD^T\\
CA^T+DB^T & CC^T+DD^T
\end{array}\right).
\end{eqnarray*}
Recalling that the standard inner product on matrices is given by
$$\langle A,B \rangle:=\mathrm{tr}(A B^T)$$
and abbreviating by
$$||A||=\sqrt{\mathrm{tr}(AA^T)}$$
its norm{,} we deduce 
\begin{equation}\label{p1}
-\mathrm{tr}(\Phi)=||A||^2+||B||^2+||C||^2+||D||^2.
\end{equation}
Using the last equality in (\ref{sympl}) and the Cauchy-Schwarz inequality we further have the inequality
\begin{equation}\label{p2}
n=\mathrm{tr}(\mathrm{id})=\langle A, D\rangle-\langle B,C \rangle
\leq \tfrac{1}{2}(||A||^2+||B||^2+||C||^2+||D||^2).
\end{equation}
Combining (\ref{p1}) and (\ref{p2}) we get
$$|\mathrm{tr}(\Phi)|\geq 2n.$$
This proves the lemma. \hfill $\square$
\\ \\
Recall that a symplectic linear map is called \emph{elliptic} if all its eigenvalues lie on the unit circle and
are different from $\pm 1$. Note that $\Phi$ is elliptic if and only if $-\Phi$ is elliptic so that the notion
of ellipticity makes sense in the projective symplectic group $\mathrm{PSp}(n)$. As an immediate Corollary we obtain
\begin{cor}
Assume that $[\Phi] \in \mathrm{PSp}(n)$ is elliptic. Then it does not lie in the unit ball
$B_\mathbb{P}$. In particular, the kei word metric thus constructed is non-trivial.
\end{cor}

\section{Involutions and diffeomorphisms}

Let $M$ be a closed, connected manifold. Abbreviate by
$\mathrm{Diff}_0(M)$ the group of smooth diffeomorphisms of $M$ isotopic to the identity. This group is simple 
\cite{thurston}. Let $R \colon M \to M$ be a smooth involution different from the identity. We do not require that
$R$ is isotopic to the identity. Note that
$$R \mathrm{Diff}_0(M) R=\mathrm{Diff}_0(M)$$
and $R$ does not commute with every element of $\mathrm{Diff}_0(M)$. Therefore we get the kei word metric
$$d_R \colon \mathrm{Diff}_0(M) \times \mathrm{Diff}_0(M) \to \mathbb{N}_0.$$
For diffeomorphisms there is an easy criterion to detect elements outside the unit ball $B_R$. Recall from
Proposition~\ref{nontri}  that an element in the unit ball is reversible. Hence to detect elements outside
the unit ball we just have to find nonreversible elements. This can be done by using the \emph{spectrum} of a diffeomorphism $\phi \colon M \to M$ which is a set of nonzero complex numbers
$$\mathrm{Spec}(\phi)\subset \mathbb{C}^*:=\mathbb{C} \setminus \{0\}$$
defined as follows. Abbreviate
$$\mathrm{Fix}(\phi):=\{x \in M: \phi(x)=x\} \subset M$$
the fixed point set of $M$. At a fixed point $x$ of $\phi$ its differential
$$d \phi(x) \colon T_x M \to T_x M$$
is  an invertible  linear map from a fixed vector space to itself in particular it has a spectrum
$$\mathrm{Spec}(d \phi(x)) \subset \mathbb{C}^*.$$
Now set
$$\mathrm{Spec}(\phi):=\bigcup_{x \in \mathrm{Fix}(\phi)} \mathrm{Spec}(d \phi(x)).$$
Because $T_x M$ is a real vector space it holds that
$\mathrm{Spec}(d \phi(x))$ is invariant under complex conjugation for each fixed point $x$ of $\phi$. In particular,
$\mathrm{Spec}(\phi)$ is invariant under complex conjugation. However, if the diffeomorphism $\phi$ is reversible its spectrum is invariant under a second involution.
\begin{lemma}\label{conju}
Suppose that $\phi \colon M \to M$ is a reversible diffeomorphism, then the spectrum $\mathrm{Spec}(\phi)$ is invariant under the involution $z \mapsto \tfrac{1}{z}$.
\end{lemma}
\textbf{Proof: }
Because $\phi$ is reversible there exists another diffeomorphism $\psi$ satisfying
\begin{equation}\label{conj}
\psi^{-1}\phi \psi=\phi^{-1}.
\end{equation}
Suppose that $x \in \mathrm{Fix}(\phi)$, i.e.,
$$x=\phi(x)$$ 
or equivalently
$$x=\phi^{-1}(x).$$
Using (\ref{conj}),  we can rewrite this equation as
$$x=\psi^{-1} \phi \psi(x)$$
or equivalently
$$\psi(x)=\phi \psi (x)$$
implying that
$$\psi(x) \in \mathrm{Fix}(\phi).$$
Differentiating (\ref{conj}) we obtain
$$d \psi^{-1}(\psi(x)) d\phi(\psi(x)) d\psi(x)=d \phi^{-1}(x)$$
which shows that $d\phi^{-1}(x)$ and $d \phi(\psi(x))$ are conjugated linear isomorphisms so that we have
$$\mathrm{Spec}(d \phi^{-1}(x))=\mathrm{Spec}(d\phi(\psi(x))).$$
If we abbreviate by
$$\rho \colon \mathbb{C}^* \to \mathbb{C}^*, \quad z \mapsto \tfrac{1}{z}$$
it holds that
$$\rho\big(\mathrm{Spec}(d\phi^{-1}(x))\big)=\mathrm{Spec}(d \phi(x))$$
so that we obtain
$$\rho\big(\mathrm{Spec}(d \phi(x))\big)=\mathrm{Spec}(d \phi(\psi(x))).$$
In particular,
$$\rho\big(\mathrm{Spec}(\phi)\big)=\mathrm{Spec}(\phi).$$
This proves the lemma. \hfill $\square$
\\ \\
In view of Proposition~\ref{nontri} the lemma has the following immediate corollary:
\begin{cor}\label{t01}
Suppose that $\phi \in \mathrm{Diff}_0(M)$ has the property that its spectrum
$\mathrm{Spec}(\phi)$ is \emph{not} invariant under the involution $z \mapsto \tfrac{1}{z}$.
Then $\phi$ lies outside of the unit ball $B_R$.
\end{cor}
We deduce from this Corollary further
\begin{cor}\label{t1}
The kei word metric $d_R \colon \mathrm{Diff}_0(M) \times \mathrm{Diff}_0(M) \to \mathbb{N}_0$ is nontrivial. 
\end{cor}
\textbf{Proof of Corollary~\ref{t1}: }Given a diffeomorphism $\phi \in \mathrm{Diff}_0(M)$ with at least one fixed point
we can after 
an arbitrarily small perturbation assume that its spectrum is not invariant under the involution $z \mapsto \tfrac{1}{z}.$ Now the assertion follows from Corollary~\ref{t01}. \hfill $\square$ 

\begin{rem}By \cite{burago-ivanov-polterovich}, there exist interesting manifolds, such as spheres of all dimensions $\mathbb{S}^{n}, n \ge 1$ and all three-dimensional closed connected manifolds, such that any bi-invariant metric on the identity component of the group of diffeomorphisms has finite diameter. For such a manifold $M$, since the kei metric is in particular a bi-invariant metric on $\mathrm{Diff}_0(M)$, its diameter is finite for any orientation-preserving or orientation-reversing diffeomorphism on $M$.
\end{rem}

\section{Symplectic involutions and Hamiltonian diffeomorphisms}

Suppose that $(M,\omega)$ is a closed, connected  symplectic manifold. Given a smooth function $H \colon M \to \mathbb{R}$ we associate to it its Hamiltonian vector field implicitly by the equation
$$dH=\omega(\cdot, X_H).$$
If $H \colon M \times [0,1] \to \mathbb{R}$ is a smooth time dependent function we abbreviate for each $t \in [0,1]$
$$H_t:=H(\cdot, t) \in C^\infty(M)$$
and abbreviate by
$$\phi^t_H:=\phi^t_{X_H}$$
the flow of the time-dependent Hamiltonian vector field $X_{H_t}$. We set
$$\phi_H:=\phi^1_H$$
the time-one map of the flow. A straightforward application of Cartan's formula shows that
$$\phi_H^* \omega=\omega$$
so that 
$$\phi_H \in \mathrm{Diff}(M,\omega)$$ 
is a symplectomorphism, i.e., a diffeomorphism which preserves the symplectic structure. We abbreviate by
$$\mathrm{Ham}(M,\omega) < \mathrm{Diff}(M,\omega)$$
the subgroup of Hamiltonian diffeomorphisms, namely all symplectomorphisms which can be obtained as the time-one map of a time-dependent Hamiltonian vector field.  By a theorem of Banyaga \cite{banyaga} the group
of Hamiltonian diffeomorphisms is simple. 

A symplectic involution is a smooth involution $R \in \mathrm{Diff}(M)$ which satisfies
$$R^* \omega=\omega.$$ 
In particular, the Hamiltonian vector field $X_H$ is invariant under $R$, i.e., 
$$R^* X_H=X_H$$
from which it follows that
$$R \mathrm{Ham}(M,\omega) R=\mathrm{Ham}(M,\omega).$$
Furthermore, if $R$ is different from the identity, then not every Hamiltonian diffeomorphism commutes with $R$, so that the symplectic involution $R$ induces a kei word metric
$$d_R \colon \mathrm{Ham}(M,\omega) \times \mathrm{Ham}(M,\omega) \to \mathbb{N}_0.$$
Again we want to find a criterion which guarantees that a given Hamiltonian diffeomorphism lies outside
the unit ball of the kei metric. The criterion of Corollary~\ref{t01} is still valid, however it is useless.
Indeed, if $x$ is a fixed point of a symplectomorphisms
$\phi$, then its differential $d\phi(x) \colon T_x M \to T_x M$ is a linear symplectic map and in particular, its spectrum
is invariant under the involution $z \mapsto \tfrac{1}{z}$. However, instead of looking at the spectrum of a Hamiltonian symplectomorphism we can look at its action spectrum. For that we have to assume that our  symplectic manifold  $(M,\omega)$ is a closed  and \emph{symplectically aspherical} namely
$\omega$ vanishes on $\pi_2(M)$. The fact that we can associate an action spectrum for fixed points of Hamiltonian diffeomorphisms hinges on the fact that we can interpret fixed points in this set-up variationally as critical points of an action functional. However, the fact that the action spectrum is well defined is a highly nontrivial theorem which
is based on Floer homology \cite{floer} and is due to Schwarz, see \cite{schwarz}.

Here is a summary of the construction. If $H \in C^\infty(M \times S^1)$ is a Hamiltonian depending periodically on time,
a \emph{periodic orbit} $x \in C^\infty(S^1,M)$ is a solution of the ODE
$$\partial_t x(t)=X_{H_t}(x(t)).$$
Periodic orbits are in one to one correspondence with fixed points of $\phi_H$ via the evaluation map $x \mapsto x(0)$. 
Contractible periodic orbits can be detected variationally as follows. Let
$$\mathcal{L} \subset C^\infty(S^1,M)$$
be the component of contractible loops in the free loop space of $M$. The action functional of classical mechanics
$$\mathcal{A}_H \colon \mathcal{L} \to \mathbb{R}$$
associates to a contractible loop $x \in \mathcal{L}$ the value
$$\mathcal{A}_H(x)=\int_D \overline{x}^* \omega-\int_0^1 H(x(t),t)dt$$
where $\overline{x}$ is a filling disk for $x$, namely a smooth map from the closed unit disk 
$D=\{z \in \mathbb{C}:|z| \leq 1\}$ to $M$ satisfying
$$\overline{x}(e^{2 \pi i t})=x(t).$$
Note that by the assumption that $x$ is contractible such a filling disk exists. Observe further that by the assumption
that the symplectic manifold $(M,\omega)$ is symplectically aspherical, the value of the action functional 
$\mathcal{A}_H(x)$ is independent of the choice of the filling disk. Abbreviate by $\mathcal{P}_H$ the set of all
contractible periodic orbits of the Hamiltonian vector field of $H$. Then it holds that
$$\mathcal{P}_H=\mathrm{crit}(\mathcal{A}_H),$$
i.e., contractible periodic orbits correspond precisely to critical points of the action functional $\mathcal{A}_H$.
Define the action spectrum of $\mathcal{A}_H$ by
$$\mathrm{ASpec}(\mathcal{A}_H):=\bigcup_{x \in \mathrm{crit}(\mathcal{A}_H)} \mathcal{A}_H(x) \subset \mathbb{R}.$$
The action spectrum of $\mathcal{A}_H$ depends on the choice of the Hamiltonian. Indeed, one can always add constants to
the Hamiltonian. This does not change the Hamiltonian vector field, but this does change the action. In order to remove this ambiguity, Schwarz normalizes the Hamiltonian by requiring that for every $t \in S^1$ one has
$$\int_M H_t \omega^{n}=0.$$
The theorem of Schwarz, see \cite[Theorem 1.1]{schwarz} now tells us
\begin{thm}[Schwarz]
For a normalized Hamiltonian the action spectrum only depends on the time one map
$\phi_H$ of the Hamiltonion flow of $H$.
\end{thm}
In view of the Theorem of Schwarz we can set for any $\phi \in \mathrm{Ham}(M,\omega)$
$$\mathrm{ASpec}(\phi)=\mathrm{ASpec}(\mathcal{A}_H)$$
where $H$ is any normalized Hamiltonian satisfying $\phi_H=\phi$.
We need the following lemma. 
\begin{lemma}\label{asp}
The action spectrum has the following properties.
\begin{description}
 \item[(i)] For every $\phi \in \mathrm{Ham}(M,\omega)$ it holds that .
 $$\mathrm{ASpec}(\phi)=-\mathrm{ASpec}(\phi^{-1}).$$
 \item[(ii)] If $\phi \in \mathrm{Ham}(M,\omega)$ and $\psi \in \mathrm{Symp}(M,\omega)$, then
 $$\mathrm{ASpec}(\phi)=\mathrm{ASpec}(\psi^{-1} \phi \psi).$$ 
 \item[(iii)] If $H$ is a $C^2$-small autonomous, i.e., time independent Hamiltonian, then
 $$\mathrm{ASpec}(\phi_H)=-\mathrm{ASpec}(H)$$
 where the action spectrum of $H$ is given by
 $$\mathrm{ASpec}(H)=\bigcup_{x \in \mathrm{crit}(H)} H(x).$$
\end{description}
\end{lemma}
\textbf{Proof:} If $\phi=\phi_H$ for a time dependent Hamiltonian $H$, set
$$H^-_t:=-H_t \circ \phi^t_H.$$
Then
$$\phi_H^{-1}=\phi_{H^-}.$$
Moreover, because $\phi^t_H$ is symplectic it follows that $H^-_t$ is normalized for
every $t \in S^1$. Note that $x \in \mathrm{crit}(\mathcal{A}_H)$ iff
$x^- \in \mathrm{crit}(\mathcal{A}_{H^-})$ where
$$x^-(t):=x(-t), \quad t \in S^1.$$
Because 
$$\mathcal{A}_H(x)=-\mathcal{A}_{H^-}(x^-)$$
assertion (i) follows. Assertion (ii) follows similarly, by observing that
the time one map of the flow of the Hamiltonian vector field of the Hamiltonian
$$H^\psi_t:=H_t \circ \psi$$
coincides with $\psi^{-1} \phi_H \psi$. Finally (iii) follows from the observation
that for an autonomous $C^2$-small Hamiltonian $H$ periodic orbits just correspond
to critical points of $H$. In particular, all periodic orbits are contractible and
of $x \in \mathrm{crit}(H)$ its action is given by
$$\mathcal{A}_H(x)=-H(x).$$
This proves (iii) and the lemma follows. \hfill $\square$
\\ \\
We are now in position to state the main result of this section.
\begin{thm}\label{main}
Suppose that $(M,\omega)$ is a closed symplectically aspherical manifold. If $\phi$ lies in 
the unit ball $B_R$ for a symplectic involution $R$, then its action spectrum satisfies
$$\mathrm{ASpec}(\phi)=-\mathrm{ASpec}(\phi).$$
\end{thm}
\textbf{Proof of Theorem~\ref{main}: } If $\phi$ lies in the unit ball $B_R$ we deduce from Proposition~\ref{nontri} that $\phi$ is conjugated to its inverse via a symplectic involution $\psi$.
Hence from properties (i) and (ii) of Lemma~\ref{asp} we deduce that
$$-\mathrm{ASpec}(\phi)=\mathrm{ASpec}(\phi^{-1})=\mathrm{ASpec}(\psi^{-1} \phi \psi)
=\mathrm{ASpec}(\phi).$$
This proves the theorem. \hfill $\square$
\\ \\
As a Corollary of this theorem we deduce that the kei metric is nontrivial.
\begin{cor}\label{t2}
Suppose that $(M,\omega)$ is a closed symplectically aspherical manifold and $R$ is a symplectic involution. Then the kei word metric
$d_R \colon \mathrm{Ham}(M,\omega) \times \mathrm{Ham}(M,\omega) \to \mathbb{N}_0$ is nontrivial.
\end{cor}
\textbf{Proof of Corollary~\ref{t2}: } Choose a normalized $C^2$-small Hamiltonian $H$ which satisfies
\begin{equation}\label{mami}
\mathrm{ASpec}(H) \neq -\mathrm{ASpec}(H).
\end{equation}
By property (iii) of Lemma~\ref{asp} we infer that
\begin{eqnarray*}
\mathrm{ASpec}(\phi_H) \neq
-\mathrm{ASpec}(\phi_H).
\end{eqnarray*}
Now the assertion of the Corollary follows from Theorem~\ref{main}. 
\hfill $\square$

\section{Quasimorphisms}

Suppose that $G$ is a group. A \emph{quasimorphism}
$$\mu \colon G \to \mathbb{R}$$
is a map for which there exists a constant $C$ with the property that
for every $g,h \in G$ the inequality
$$\big|\mu(gh)-\mu(g)-\mu(h)\big| \leq C$$
holds true. In other words a quasimorphisms is up to bounded error a homomorphisms from
the group $G$ to the real numbers. A quasimorphism is called \emph{homogeneous} if for every
$g \in G$ and $k \in \mathbb{Z}$ it holds that
$$\mu(g^k)=k\mu(g).$$
Given any quasimorphism $\mu$ its \emph{homogenization} for $g \in G$ is defined as
$$\overline{\mu}(g)=\lim_{k \to \infty} \frac{1}{k}\mu(g^k).$$
The homogenization of a quasimorphism is then a homogeneous quasimorphisms
\cite[Prop.\,3.3.1]{bavard}. Note that if a 
quasimorphism is already homogeneous then it coincides with its homogenization. We say that
a quasimorphism $\mu$ is \emph{nontrivial} if its homogenization is not the zero map, in which
case by homogeneity it is necessarily unbounded. 

We assume now that we are in the situation of Section~\ref{special}, namely
$G < \mathfrak{S}(X)$ is a simple subgroup of the permutation group of a set $X$ and $R \in 
\mathrm{Inv}(X)$ is an involution satisfying $RGR=G$ and which does not commute with every
element of $G$. The involution $R$ gives rise to an
involutative automorphism
$$I_R \colon G \to G, \quad g \mapsto RgR.$$
The main theorem of this section is the following result.
\begin{thm}\label{thm: alg qm}
Suppose that there exists a nontrivial $I_R$-invariant quasimorphism $\mu \colon G \to \mathbb{R}$.
Then the kei metric $d_R$ on $G$ has infinite diameter. 
\end{thm}
\textbf{Proof: } Since the homogenization of an $I_R$-invariant quasimorphism is itselt
$I_R$ invariant we can assume without loss of generality that $\mu$ is homogeneous. We argue
by contradiction and suppose that the diameter of $d_R$ on $G$ is bounded. This means
that there exists $N \in \mathbb{N}$ such that for every $g \in G$ there exist 
$g_1,\ldots,g_n,h_1,\ldots,h_n \in G$ for $n \leq N$ such that
$$g=R_{g_1}R_{h_1}R_{g_2}R_{h_2}\cdots R_{g_n}R_{h_n}.$$
For $1 \leq i \leq n$ we abbreviate
$$S_i=R_{g_i}R_{h_i}.$$
Since $R_{g_i}$ and $R_{h_i}$ are involutions we have
$$S_i^{-1}=R_{h_i}S_i R_{h_i}$$
or equivalently
$$(h_i^{-1}S_ih_i)^{-1}=R h_i^{-1} S_i h_i R.$$
Since $\mu$ is an $I_R$-invariant homogeneous quasimorphism we have
\begin{eqnarray*}
\mu\big(h_i^{-1} S_i h_i\big)=\mu\big(R h_i^{-1} S_i h_i R\big)=\mu\big((h_i^{-1}S_i h_i)^{-1}\big)
=-\mu\big(h_i^{-1} S_i h_i\big)
\end{eqnarray*}
and therefore
$$\mu\big(h_i^{-1} S_i h_i \big)=0, \quad 1 \leq i \leq n.$$
Since homogeneous quasimorphisms are invariant under conjugation (see for instance \cite{simon-salamon}) we have
$$\mu(S_i)=0, \quad 1 \leq i \leq n.$$
Since $g=S_1S_2 \ldots S_n$ we obtain from the quasimorphism property
$$\mu(g) \leq Cn \leq CN,$$
which shows that $\mu$ is uniformly bounded. In particular, since $\mu$ is homogeneous it
has to be the zero map which contradicts the assumption that it is nontrivial. This proves the theorem. \hfill $\square$

This theorem gives a useful criteria to show a kei metric to have infinite diameter. 

In the forthcoming work \cite{FPSZ} we introduce the grand kei metric. While the kei metric deals with involutions, therefore the reversibility here is \emph{strong}, \emph{i.e.} a map is conjugate to its inverse by an involution. The grand kei metric deals with a weaker version of reversibility, in which the conjugation map is not required to be an involution. Therefore if the diameter with respect to the grand kei metric is infinite, then it is so for kei metrics. Using quasi-homomorphisms we will show in \cite{FPSZ} that the grand kei metric on the group of Hamiltonian diffeomorphisms of any closed connected smooth surface has infinite diameter. This then shows that the kei metrics on the group of Hamiltonian diffeomorphisms for any symplectic or anti-symplectic involution have infinite diameter.
\appendix

\section{Appendix: The word metric associated to N-periodic maps}
We explain in this section that the construction of a word metric from a kei of involutions can be generalized to N-periodic maps.

Again let $X$ be a set and denote by $\mathrm{Per}^{N}(X)$ the set of all N-periodic maps on $X$, i.e., the set of all
maps $R \colon X \to X$, satisfying $R^N=\mathrm{id}|_X$. Conjugation defines a product on
$\mathrm{Per}^{N}(X)$
$$R \star S= R S R^{-1}.$$
This product has the following properties
\begin{description}
 \item[(i)] $R \star R=R.$
 \item[(ii)] $\underbrace{R \star (R \star \cdots  \star (R}_\text{N-times} \star S))=S.$
 \item[(iii)] $R \star (S \star T)=(R \star S) \star (R \star T).$
\end{description}
Let's mimic the terminology of Takasaki and call this an \emph{N-kei}. 

\begin{rem}These are examples of quandles (\cite{joyce}, \cite{carter}). Recall that a quandle is a set $Q$ together with a multiplication $\star$ such that
\begin{itemize}
\item $a \star a=a$;
\item $a \star (b \star c) =(a \star b) \star (a \star c), \forall a, b, c \in Q$;
\item For all $a, c \in Q$, there exits a $b \in Q$, such that $a \star b = c$. 
\end{itemize}
\end{rem}

Given an N-kei $(\mathcal{K},\star)$ we
get by property (ii) a map
$$\mathcal{L} \colon \mathcal{K} \to \mathrm{Per}^{N}(\mathcal{K}), \quad R \mapsto \mathcal{L}_R$$
where $\mathcal{L}_R$ is the left multiplication of $R$ on $\mathcal{K}$
$$\mathcal{L}_R \colon \mathcal{K} \to \mathcal{K}, \quad S \mapsto R \star S.$$
If $R, S, T \in \mathcal{K}$ are three elements of our N-kei, then again we have
\begin{eqnarray*}
\mathcal{L}_{R} \mathcal{L}_{S} \mathcal{L}_{R^{-1}}(T)&=&
R \star (S \star (R^{-1} \star T))\\
&=&(R \star S) \star (R \star (R^{-1} \star T))\\
&=&(R \star S) \star T\\
&=&\mathcal{L}_{R \star S}(T).
\end{eqnarray*}
 Thus the image of the map $\mathcal{L}$ is a sub N-kei of $\mathrm{Per}^{N}(\mathcal{K})$.

For a set $X$, We have denote by
$$\mathfrak{S}(X):=
\{\phi \colon X \to X, \phi\,\,\mathrm{bijective}\}$$
the group of permutations of $X$. 
For a N-kei $\mathcal{K}$, we define the \emph{transvection group}
$$G(\mathcal{K}) \subset \mathfrak{S}(\mathcal{K})$$
as
$$G(\mathcal{K})=\langle \mathcal{L}_R \mathcal{L}_S^{-1}=\mathcal{L}_R \mathcal{L}_{S^{-1}}=\mathcal{L}_R \mathcal{L}_S^{N-1}: R,S \in 
\mathcal{K} \rangle.$$

It is direct to verify that the generating set of the transvection group consisting of all products $\mathcal{L}_R \mathcal{L}_S^{-1}$
is symmetric and is conjugation invariant. Therefore it defines a bi-invariant metric
$$d_\mathcal{K} \colon G(\mathcal{K}) \times G(\mathcal{K}) \to \mathbb{N}_0$$
which we refer to as the \emph{N-kei word metric} on the transvection group.

Suppose that for a set $X$ 
$$G < \mathfrak{S}(X)$$
is a subgroup of the permutation group of $X$
and 
$$R \in \mathrm{Per}^{N}(X)$$
is an N-periodic map satisfying
$$RGR^{-1}=G,$$
i.e., $R$ is  not necessarily an element of $G$ but every element of
$G$ when conjugated with $R$ in $\mathfrak{S}(X)$ is again in $G$.
For $g \in G$ we again set
$$R_g:=g R g^{-1}.$$ 
and
$$\mathcal{K}_{R,G}:=\{R_g: g \in G\}.$$
It is again direct to verify
\begin{lemma}
$\mathcal{K}_{R,G}$ is a sub N-kei of $\mathrm{Per}^{N}(X)$. 
\end{lemma}

We may again define the canonical homomorphism
$$\Psi \colon G \to \mathfrak{S}(\mathcal{K}_{G,R})$$
given for $g,h \in G$ by
$$\Psi(g)(R_h)=R_{gh},$$
and the subgroup
$$G_R < \mathfrak{S}(X)$$
which is generated by all transvections of elements in the N-kei
$\mathcal{K}_{R,G}$, namely
$$G_R=\langle R_g R_h^{-1}: g,h \in G\rangle.$$
We again have
\begin{lemma}\label{ngr}
The group $G_R$ is a normal subgroup of $G$ and the homomorphism $\Psi$ maps $G_R$ to the 
transvection group $G(\mathcal{K}_{G,R})$.
\end{lemma}
and
 
\begin{prop}\label{ntra}
Assume that the group $G$ is simple and $R$ is not commuting with every element of $G$. Then it holds that $G=G_R$ and $\Psi$ maps
$G_R$ isomorphically to the transvection group $\mathcal{K}_{G,R}$. In particular, 
$G$ is canonically isomorphic to $\mathcal{K}_{G,R}$. 
\end{prop}

Now assume that $G<\mathfrak{S}(X)$ is a simple group. If $R \in \mathrm{Per}^{N}(X)$ is an N-periodic map such that
$RGR^{-1}=G$ but $R$ does not commute with every element of $G$, by Proposition~\ref{ntra}, we can identify $G$ with the transvection group of the N-kei $\mathcal{K}_{G,R}$. In particular, $R$ endows $G$ with a bi-invariant  N-kei word metric 
$$d_R:=d_{\mathcal{K}_{R,G}}\colon G \times G \to \mathbb{N}_0.$$

In particular, this N-kei word metric can be defined for $\mathrm{Diff}_0(M)$ of a closed, connected manifold by a smooth N-periodic map different from the identity, or for $\mathrm{Ham}(M,\omega)$ of a closed, connected symplectic manifold by a smooth (anti-)symplectic N-periodic map different from the identity.
 
Nevertheless, in contrast to the case of involutions, now we do not have an analogue of Proposition \ref{nontri} for general N-periodic maps, therefore it is not clear now if the N-kei word metric $d_{R}$ is trivial or not. We leave this as an open question.

\medskip
\medskip
\medskip
\medskip
\medskip
\medskip
\medskip
\medskip
\medskip

\noindent
Urs Frauenfelder, Institute of Mathematics, University of Augsburg, Germany, Email: urs.frauenfelder@math.uni-augsburg.de \\
\\
Lei Zhao, School of Mathematical Sciences, Dalian University of Technology,China, Email: zhao1899@dlut.edu.cn

\end{document}